\documentclass[12pt]{article}
\usepackage[margin=1in]{geometry}
\usepackage{amsthm, amsmath, amsfonts, amssymb}
\usepackage{graphicx}
\usepackage{authblk}

\newcommand{\Z}[0]{\mathbb{Z}}

\newcommand{\C}[0]{\mathbb{C}}
\newcommand{\F}[0]{\mathbb{F}}

\newtheorem{claim}{Claim}
\newtheorem{thm}[claim]{Theorem}
\newtheorem{lemma}[claim]{Lemma}
\newtheorem{prop}[claim]{Proposition}
\newtheorem{cor}[claim]{Corollary}
\theoremstyle{definition}
\newtheorem{definition}{Definition}

\theoremstyle{remark}
\newtheorem*{rem}{Remark}

\theoremstyle{definition}
\newtheorem{example}{Example}[section]

\title{Uniform Recurrence in the Motzkin Numbers and Related Sequences mod $p$}
\author{Nadav Kohen}
\affil{Indiana University\\Bloomington, IN, USA\\nkohen@iu.edu}
\date{February 22, 2024}

\begin{document}
\maketitle
\begin{abstract}
Many famous integer sequences including the Catalan numbers and the Motzkin numbers can be expressed in the form $ConstantTermOf\left[P(x)^nQ(x)\right]$ for Laurent polynomials $Q$, and symmetric Laurent trinomials $P$. In this paper we characterize the primes for which sequences of this form are uniformly recurrent modulo $p$. For all other primes, we show that $0$ has density $1$. This will be accomplished by showing that the study of these sequences mod $p$ can be reduced to the study of the generalized central trinomial coefficients, which are well-behaved mod $p$.
\end{abstract}

\section{Introduction}
The Motzkin numbers (A001006 of \cite{oeis}), $M_n$, count the number of lattice paths from the origin to $(n, 0)$ which do not go below the x-axis with steps $U = (1, 1), L = (1, 0)$, and $D = (1, -1)$. See \cite{motzkinsurvey} for many other combinatorial settings in which the Motzkin numbers arise. Some work has been done to characterize $M_n$ and similar sequences modulo various prime powers. For example, Deutsch and Sagan \cite{deutschsagan} characterized $M_n\mod 3$ as well as showing for which $n$, $M_n\equiv 0$ mod $p$ for $p=2,4,5$.\\

In recent years, much of this work has utilized Rowland and Zeilberger's finite automaton \cite{rowlandzeilberger} which encodes the behavior of any sequence of the form $ct\left[P^nQ\right]\mod p$ where $ct$ stands for ``constant term of,'' and $P$ and $Q$ are Laurent polynomials (possibly in multiple variables). Burns \cite{smallprimesmotzkin} has used these automata to study the asymptotic behavior of $M_n$ mod small primes. And Rampersad and Shallit \cite{CongruenceAutomaton} have used these automata alongside the automatic theorem prover Walnut \cite{walnut} (finite state automata have a decidable first-order theory) to re-prove Deutsch and Sagan's results, as well as showing that $M_n\mod 5$ is uniformly recurrent (see Definition \ref{uf}), and various other congruence properties of the Motzkin numbers, Catalan numbers, and central trinomial coefficients.\\

The central trinomial coefficients (A002426 of \cite{oeis}) are $T_n = ct\left[(x^{-1} + 1 + x)^n\right]$, and are related to the Motzkin numbers since $M_n = ct\left[(x^{-1} + 1 + x)^n(1 - x^2)\right]$. From this, one can derive that $2M_n = 3T_n + 2T_{n+1} - T_{n+2}$ (see section \ref{AppSeq}). In this paper, we call sequences of the form $a_n = ct\left[(\alpha_{-1}x^{-1} + \alpha_0 + \alpha_1x)^n\right]$ \emph{generalized central trinomial coefficients}.\\

In Problem 6 of \cite{CongruenceAutomaton}, Rampersad and Shallit ask for a characterization the primes, $p$, for which the Motzkin numbers mod $p$ are uniformly recurrent (Problem 6 in \cite{CongruenceAutomaton}). Based on Burns' results for small primes in \cite{smallprimesmotzkin}, Rampersad and Shallit conjecture that the answer will be the sequence A113305 of \cite{oeis} of primes not dividing any central trinomial coefficient, and that for all other $p$, the Motzkin numbers mod $p$ will be $0$ with density $1$.\\

In this paper, we confirm these conjectures (Theorem \ref{main}) by showing that the set of primes, $p$, for which any sequence which arises as an integral linear combination of generalized central trinomial coefficients are uniformly recurrent mod $p$ is the set of primes which do not divide any of the corresponding generalized central trinomial coefficients. In particular, the Motzkin numbers are of this required form. Furthermore, we confirm that for all other $p$, sequences of this form are $0$ mod $p$ with density $1$.\\

In the case that $p$ does divide a central trinomial coefficient, our approach is to utilize the fact that generalized central trinomial coefficients mod $p$ are determined independently by the digits in their base $p$ expansions (Proposition \ref{lucas}) and thus $p$ divides one of the first $p$ coefficients and also any index whose base $p$ expansion contains that digit. This forces the set of indices for which $p$ divides that central trinomial coefficient to have density $1$. This also forces there to be arbitrarily long runs of $0$s which inhibits uniform recurrence.\\

In the case that $p$ does not divide any central trinomial coefficient, our approach is again to utilize Proposition \ref{lucas} to see that in any integral linear combination, the prefix of base $p$ digits which all of the indices $n,n+1,\ldots,n+h$ have in common can be factored out. And since individual central trinomial coefficients mod $p$ recur within a constant bound (Lemma \ref{length_one}), in most cases we can force any word in our general sequence to recur by adding only to the prefixes shared by all indices of central trinomial coefficients involved.\\

Lastly, in section \ref{AppSeq} we show that sequences of the form $b_n = ct\left[P(x)^nQ(x)\right]$, where $P$ is a symmetric Laurent trinomial and $Q$ is any Laurent polynomial, can be written as combinations of $ct\left[P(x)^{n+i}\right]$ for various $i$, so that our results apply to all sequences of this form. If we let $V_k$ denote the $k$-dimensional irreducible representation of $SU(2,\C)$, then the number of irreducible components of dimension $d$ in $\left(V_1^{m_1}\oplus V_2^{m_2}\right)^{\otimes n}$ yields a sequence, $b_n^{d,m_1,m_2}$, of the above form for every $d,m_1$, and $m_2$ in $\Z_{\geq 0}$. Namely, $b_n^{d,m_1,m_2} = ct\left[P(x)^nQ(x)\right]$ where $P(x) = m_2x^{-1} + m_1 + m_2x$ and $Q(x) = x^{d-1} - x^{d-3}$. For example, $b_n^{1,1,1}$ are the Motzkin numbers and $b_n^{1,2,1}$ are the Catalan numbers.

\subsection{Notation and Conventions}
Throughout this paper, $P(x)$ denotes a Laurent trinomial of the form $\alpha_{-1}x^{-1} + \alpha_0 + \alpha_1x$ with $\alpha_i\in\Z$, $Q(x)$ denotes an arbitrary Laurent polynomial with integer coefficients, and $ct\left[Q(x)\right]$ denotes the constant term of $Q(x)$. For a fixed $P(x)$, we let $a_n$ denote the sequence $ct\left[P(x)^n\right]$.\\

If $\Sigma$ is a set, $\Sigma^*$ denotes the set of words (i.e. strings) of any length whose characters are from $\Sigma$ (including the empty word). If $n$ is a non-negative integer, and $p$ is a prime then let $n_p\in\F_p^*$ be the word whose characters are the digits of $n$ in base $p$. That is, if we let $n_p[i]$ denote the $i$th digit in the base $p$ expansion of $n$ so that $n = \sum_{i\in\Z_{\geq 0}} n_p[i]p^i$, then $n_p = (n_p\left[\lfloor\log_p n\rfloor\right])\cdots(n_p[1])(n_p[0])$. Note that when working with strings, exponents are used to denote repetition, for example, $(p-1)^k$ denotes a run of $k$ characters that are all the character $(p-1)$. Also note that any statement made in this paper about $n_p$ should also hold for $0^kn_p$ for any $k$.

This paper is primarily focused on showing when sequences mod primes are uniformly recurrent, which can be thought of as a weaker form of periodicity:
\begin{definition}\label{uf} A sequence $s_n$ is called \emph{uniformly recurrent} if for every word (i.e. contiguous subsequence) $w = s_is_{i+1}\cdots s_{i+\ell-1}$, there is a constant $C_w$ such that every occurrence of $w$ is followed by another occurrence of $w$ at distance at most $C_w$. I.e. there is a $j\leq C_w$ such that $w = s_{i+j}s_{i+j+1}\cdots s_{i + j + \ell - 1}$.
\end{definition}
\section{Central Trinomial Coefficients}
\begin{prop}\label{lucas}
For any prime $p$, the generalized central trinomial coefficients, $a_n = ct\left[P(x)^n\right]$, satisfy $a_n\equiv \prod a_{n_p[i]}\mod p$.
\end{prop}
\begin{proof}
We induct on the number of digits in $n_p$. Certainly if $n = n_p[0] < p$ then $a_n = a_{n_p[0]}$. Otherwise, if $n = qp + n_p[0]$, then 
\begin{align*}
a_n &= ct\left[P(x)^{qp + n_p[0]}\right]\\
&\equiv ct\left[P(x^p)^qP(x)^{n_p[0]}\right]\mod p &\left(P(x)^p\equiv P(x^p)\mod p\right)\\
&= ct\left[P(x^p)^q\right]ct\left[P(x)^{n_p[0]}\right] &\left(n_p[0]<p\text{ so there will be no cancellation}\right)\\
&= ct\left[P(x)^q\right]ct\left[P(x)^{n_p[0]}\right] &\left(ct\left[P(x^k)^n\right] = ct\left[P(x)^n\right]\right)\\
&= a_qa_{n_p[0]}\\
&= \prod a_{n_p[i]} &(\text{by induction, since }q_p\text{ has fewer digits than }n_p).
\end{align*}
\end{proof}

This is why the central trinomial coefficients, A002426 of \cite{oeis}, satisfy this Lucas congruence (see \cite{LucasCongruence}) since they are defined by $T_n = ct\left[(x^{-1} + 1 + x)^n\right]$. However, in the case that $\alpha_0 = 0$, we usually want to discuss the sequence $ct\left[(\alpha_{-1}x^{-1} + \alpha_1x)^{2n}\right]$ since the odd powers all have $0$ constant term. But this is no issue since $ct\left[(\alpha_{-1}x^{-1} + \alpha_1x)^{2n}\right] = ct\left[(\alpha_{-1}^2x^{-1} + 2\alpha_{-1}\alpha_1 + \alpha_1^2x)^{n}\right]$ (for example, the central binomial coefficients, A000984 of \cite{oeis}, are $B_n = ct\left[(x^{-1} + 2 + x)^n\right]$). This gives us,

\begin{cor}
If $b_n = ct\left[P(x)^{2n}\right] = a_{2n}$, then $b_n$ also satisfies the congruence $b_n\equiv \prod b_{n_p[i]}\mod p$.
\end{cor}\qed

\section{Combinations of Central Trinomial Coefficients}
Here we will characterize the primes, $p$, for which generalized central trinomial coefficients, $a_n = ct\left[P(x)^n\right]$, are uniformly recurrent mod $p$. This is nearly accomplished in \cite{CongruenceAutomaton}, but here we do away with the assumption that one of $\{a_0,a_1,\ldots,a_{p-1}\}$ needs to be a primitive root. Additionally, the proof has been extended so as to characterize the primes for which any integral linear combination of $a_{n+i}$ is uniformly recurrent mod $p$, where the characterization will be independent of the linear combination given.\\

In particular, weighted Motzkin sequences \cite{WeightedMotzkin} (including the standard Motzkin sequence) can be written as integral linear combinations of generalized central trinomial coefficients, so our results will apply to these sequences.\\

We begin with the case where our sequences will not be uniformly recurrent:

\begin{example}\label{mot3}
The Motzkin numbers satisfy $2M_n = 3T_n + 2T_{n+1} - T_{n+2}$ where $T_n = ct\left[(x^{-1} + 1 + x)^n\right]$ (see section \ref{AppSeq} or \cite{motzkincentralid}). For $p>2$, $M_{p,n} = 2^{-1}(3T_n + 2T_{n+1} - T_{n+2})$ gives us a sequence congruent to $M_n$ mod $p$ where $2^{-1}$ is the multiplicative inverse of $2$ mod $p$.\\

Consider $p = 3$ so that $p\mid T_2 = 3$. Then by Proposition \ref{lucas}, $T_n\equiv 0$ any time that $n_p$ contains a $2$. This in turn implies that any time all three of $n_p, (n+1)_p, (n+2)_p$ contain a $2$, then all three of $T_n,T_{n+1}, T_{n+2}\equiv 0$ and thus $M_n\equiv 0$. Thus, to find a run of $0$s in $M_n\mod 3$ of length at least $3^{k-1}$, we can use the fact that for every integer, $n$, in $[2(3)^k, 2(3)^k + 3^{k-1}]$, all three of $n,n+1,n+2$ have a $2$ in their base $3$ representations, so long as $k>1$ (If $k = 1$ and $m = 2(3)^1 + 3^{1-1} = 2(3) + 1$, then $m+2 = 3^2$).
\end{example}

\begin{prop}\label{zeros}
If $p$ is a prime dividing some element of $a_n$, and if $b_n = \sum_{i = 0}^h c_ia_{n+i}$ where $c_i\in\Z$, then $b_n\mod p$ has arbitrarily large runs of $0$s. Thus, $b_n\mod p$ is not uniformly recurrent. In particular, these statements hold for $b_n = 1\cdot a_n$.
\end{prop}
\begin{proof}
Let $0 < z < p$ be an integer such that $a_z \equiv 0\mod p$. Because $b_n = \sum_{i = 0}^h c_ia_{n+i}\equiv \sum_{i=0}^h\left(c_i\cdot\prod a_{(n+i)_p[j]}\right)\mod p$ by Proposition \ref{lucas}, any prefix of base $p$ digits that all of the indices $n,\ldots,n+h$ share, say $a_{n_x}$ up to $a_{n_y}$, can be factored out of this sum so that we have
\begin{align*}
b_n &\equiv \sum_{i=0}^h\left(c_i\cdot\prod a_{(n+i)_p[j]}\right)\\
&= \sum_{i=0}^h\left(c_i\cdot\prod_{j=x}^y a_{n_p[j]}\prod_{j<x}a_{(n+i)_p[j]}\right)\\
&= \prod_{j=x}^y a_{n_p[j]}\left(\sum_{i=0}^h\left(c_i\cdot\prod_{j < x} a_{(n+i)_p[j]}\right)\right).
\end{align*}
In particular, we have that if $n_p[j] = z$ for any $x\leq j\leq y$, then $b_n\equiv 0\mod p$. Therefore, for sufficiently large integers $k$ (relative to $\log_p h$), $z\cdot p^k$ marks the beginning of a run of $0$s mod $p$ of length at least $p^{k-1}$ since $\left(z\cdot p^k\right)_p = z0^k$ (i.e. $z$ followed by $k$ $0$s). And $p^{k-1}$ can be made arbitrarily large.
\end{proof}

\begin{prop}\label{density}
If $p$ is a prime dividing some element of $a_n$, and if $b_n = \sum_{i = 0}^h c_ia_{n+i}$ where $c_i\in\Z$, then $0$ has density $1$ in the sequence $b_n\mod p$.
\end{prop}
\begin{proof}
As mentioned in the proof of Proposition \ref{zeros}, if any digit is $z$ (such that $p\mid a_z$) in a shared prefix of a run of indices, $n$ through $n+h$, then $b_n\equiv 0\mod p$. If $\beta = \lfloor\log_p(h)\rfloor + 1$, then consider the first $p^k$ terms of our sequence ($n=0,\ldots,p^k - 1$). For $k >\beta$, if any of the first $(k-\beta)$ digits of an index $n_p$ are $z$, then that $z$ must be part of the shared prefix (i.e. every string $n_p,\ldots,(n+h)_p$ have a $z$ in that position) and so $b_n\equiv 0$. So there are at least $p^k - (p-1)^{k-\beta}p^\beta$ of the first $p^k$ terms of $b_n$ which are divisible by $p$ (since there are $p-1$ choices for the first $k-\beta$ digits of $n_p$ which allow non-zero $b_n$), and so the proportion is at least $\frac{p^k - (p-1)^{k-\beta}p^\beta}{p^k} = 1 - \frac{p^\beta}{(p-1)^\beta}\left(\frac{p-1}{p}\right)^k\rightarrow 1$ as $k\rightarrow\infty$.
\end{proof}

This completes the characterization of what happens when $p\mid a_n$ for some $a_n$, we now turn to the case where $p\nmid a_n$ for all $n$, in which case our sequences will be uniformly recurrent. This result will usually boil down to using the fact that $a_n$ has uniform recurrence for words of length $1$:

\begin{lemma}\label{length_one}
If $p\nmid a_n$ for all $n$, then for every $n\in \Z_{\geq 0}$, there is an $n'\in\Z_{\geq 0}$ such that $n' > n$, $n' - n < p^{p^{(p-1)} + p + 1}$, and $a_n = a_{n'}$.
\end{lemma}
\begin{proof}
Given $n$, write it as a word in $\{0,1,\ldots,p-1\}^*$ via its base $p$ expansion $n_p = n^*(n_p[p^{p-1}])\cdots (n_p[1])(n_p[0])$ where the leading $n_p[i]$ may be $0$ and $n^*\in\{0,1,\ldots,p-1\}^*$ may be the empty word. We will find $n'$ by adding only to this suffix (or slightly more) to achieve the bound.

Since the value of $a_n\mod p$ is independent of the order of the $n_p[i]$ by Proposition \ref{lucas}, if there exists any $i>j$ such that $n_p[i] < n_p[j]$, then we can let $n'$ be the result of switching the $i$th and $j$th (least significant) digits of $n$.

Otherwise we have that the digits $n_p[p^{p-1}]$ through $n_p[0]$ are descending, and by the pigeonhole principle there is some $i$ such that $n_p[i] = n_p[i-1] = \cdots = n_p[i-(p-2)]$. Because $p\nmid a_{n_p[i]}$, we can apply Fermat's Little Theorem to see that the contribution of these digits is $a_{n_p[i]}^{p-1}\equiv 1\mod p$. If there is some such $i$ as above with $n_p[i]\neq p-1$, then we can let $n'$ be the result of replacing these digits in $n$ with $n_p[i]+1$ (which will result in the same contribution to the product of Proposition \ref{lucas} of $a_{n_p[i] + 1}^{p-1}\equiv 1\mod p$).

Lastly, if there are only $i$ which begin runs of $n_p[i]$ of length $p-1$ with $n_p[i] = p-1$, then the word $n_p$ is of the form $n^{**}(n_p[k + \vert\gamma\vert])(p-1)^k\gamma$ where $k > p$ and $\gamma\in\{0,1,\ldots,p-2\}^*$ has length less than $p^{p-1} - p$, and $n_p[k + \vert\gamma\vert]\neq p-1$ is the first non-$(p-1)$ digit (from the right). In this case, if $k = q(p-1) + r$ with $r < p-1$, we can let $n'$ correspond to the word $(n')_p = n^{**}(n_p[k + \vert\gamma\vert] + 1)0^{(q-1)(p-1)}(n_p[k + \vert\gamma\vert])(n_p[k + \vert\gamma\vert] + 1)^{p-2}(p-1)^r\gamma$. Again using Proposition \ref{lucas}, that $p\nmid n_p[i]$ for any $i$, and Fermat's Little Theorem, it is clear that $a_n\equiv a_{n'}\mod p$.
\end{proof}

\begin{example}
To illustrate this last case, let $p=5$ and consider $a_n = T_n = ct\left[(x^{-1} + 1 + x)^n\right]$. Let $n_p = 12324^{678}333222111000$ so that $n^{**} = 123$, $\gamma = 333222111000$, $k = 678$, $q = 169$, $r = 2$, and $n_p[k + \vert\gamma\vert] = 2$. Then $(n')_p = 12330^{672}23^34^2333222111000$. Both $n$ and $n'$ have the same number of each digit mod $p-1 = 4$ so that by Proposition \ref{lucas} and Fermat's Little Theorem, they are congruent.
\end{example}

We will now give an example to motivate the approach for the proof our main theorem.

\begin{example}
Let $p=5$, $a_n = T_n$ and $b_n = M_{p,n} = 2^{-1}(3a_n + 2a_{n+1} - a_{n+2})$ as in Example \ref{mot3}. Consider $n = 75156245$ so that $n_p = 123214444443$, $(n+1)_p = 123214444443$, and $(n+2)_p = 123220000000$. There are two ways we can construct an $n'>n$ such that $b_n\equiv b_{n'}$: We can use the fact that each of $a_n,a_{n+1},a_{n+2}$ share a factor of $a_1a_2a_3a_2 = a_{192}$ (because $192_p = 1232$) from the shared prefix of these three indices, so we can use Lemma \ref{length_one} to add some value to this shared prefix; in this case, it just so happens that $a_{192}\equiv 3\equiv a_{199}$ so we can let $n' = n + 7(5)^8$. Alternatively, we can use Fermat's Little theorem to replace some multiple of $p-1$ of the $4$s, and use those positions to undo the effect of incrementing the first non-$4$. This turns $n_p = (1232)1(4444)443$ into $(n')_p = (1232)2(2221)443$. In either case, all three pairs $(n,n'), (n+1, n'+1), (n+2,n'+2)$ have the same number of each digit mod $p-1$ so that $b_n\equiv b_{n'}$. However, we have actually accomplished more than this, in our second approach using Fermat's Little Theorem, if we let $\Delta = n' - n$ then every index from $m_p = 123214444000$ to $123220000443$ satisfies $b_m\equiv b_{m+\Delta}$. And in our first approach, the same is true for every $m_p$ from $123200000000$ through $123244444444$. We use both of these ways of getting recurring runs in $b_n$, depending on the desired length of run, in the following proof.
\end{example}

\begin{thm}\label{main}
If $a_n = ct\left[P(x)^n\right]$ where $P(x) = \alpha_{-1}x^{-1} + \alpha_0 + \alpha_1x$, and if $b_n = \sum_{i = 0}^h c_ia_{n+i}$ where $c_i\in\Z$, then $b_n$ is uniformly recurrent mod $p$ if and only if $p$ does not divide any $a_n$ (which can be checked for $n < p$).
\end{thm}
\begin{proof}
One direction is Proposition \ref{zeros}, so we need only show that if $p\nmid a_n$ for any $n$ (so that Lemma \ref{length_one} applies) then $b_n$ is uniformly recurrent mod $p$.\\

First, a proof sketch: Given a word $w = (b_i\mod p)(b_{i + 1}\mod p)\cdots(b_{i+\ell - 1}\mod p)$ of length $\ell$, we wish to bound its next occurrence by using Lemma \ref{length_one} to increment a shared prefix of all of the indices appearing on $a_\bullet$ in the expansions of the $b_\bullet$ in $w$. Our main complication occurs when our indices on $b_\bullet$ are leading up to multiples of large powers of $p$ where many digits will differ between $i$ and $i+h$. If $k$ is the largest power of $p$ which has a multiple appear in an index of $a_\bullet$, then this approach of incrementing a prefix will be permissible if $\ell$ is very large (relative to $p^k$) or if $k < p$ so that we can be assured a recurrence of $w$ within some constant (very large) multiple of $\ell$ shifts. When $\ell$ is not large enough and $k > p - 1$, we will instead use Fermat's Little Theorem, similar to how it is used in the proof of Lemma \ref{length_one}, to achieve our recurrence within the bound of a constant (very large) multiple of $\ell^{\frac{p}{p-1}}$.\\

Let $\beta = \lfloor\log_p h\rfloor + 1$, which bounds the number of digits in $h_p$. We will black-box the last $\beta$ digits of our base $p$ expansions, and add a factor of $p^\beta$ to our constant multiples of $\ell$ to simplify our argument. Hence, we will use the convention that $\gamma\in\{0,\ldots, p-1\}^\beta$ (if $h=0$, then $\beta = 0$ and $\gamma$ is the empty word), and sometimes we allow $\gamma$ to be a variable from this domain.

Let $p^s$ be the largest power of $p$ which has a multiple in the interval $(i, i+h+\ell)$. Pick $\alpha\geq 1$ such that $p^{(\alpha - 1)(p-1) + \beta} - p^\beta\leq \ell < p^{\alpha(p-1) + \beta} - p^\beta$. If $s\geq \beta + p$, let $k = s - \beta$ and $k = q(p-1) + r$ with $r < p-1$. In this case, if $\alpha\geq q$, then
\begin{align*}
p^s &= p^{k+\beta}\\
&= p^{(q-1)(p-1) + \beta}p^{p-1+r}&(k = q(p-1) + r)\\
&\leq p^{p-1+r}(\ell + p^\beta)&(q\leq \alpha\text{ and }p^{(\alpha-1)(p-1) + \beta}\leq \ell + p^\beta)\\
&\leq p^{2p}(1 + p^\beta)\ell&(r < p\text{ and }\ell + p^\beta\leq \ell + \ell p^\beta)\\
&\leq p^{3p + \beta}\ell.&(1+p^\beta\leq p^{p+\beta})
\end{align*}
And if $s < \beta + p$, then clearly $p^s\leq p^{3p + \beta}\ell$ as well. So in either case, if we add to the prefix of all $a_\bullet$ appearing in expansions of characters of $w$ (i.e. add some multiple of $p^s$) then we are guaranteed a recurrence within $C\cdot p^s\leq C\cdot p^{3p + \beta}\ell$ shifts of our original $i$, where $C$ is a constant bounding the recurrence in our prefix mod $p$ (i.e. bounding the recurrence of $a_i$ in $a_\bullet$ for all $i$ in Lemma \ref{length_one}). We get this recurrence because adding $\Delta p^s$ increments the prefix of the base $p$ expansions of all indices by $\Delta$, and we can pick a $\Delta < C$ using Lemma \ref{length_one}.\\

Otherwise, we have $s \geq p + \beta$ ($k > p - 1$) and $\alpha < q$. Let $n$ be the first index such that $n+h$ is a multiple of $p^{k+\beta}$. Let $$n_p = n^*(m)(p-1)^k\gamma$$ where $n^*\in\{0,\ldots,p-1\}^*$ and $m\in\{0,\ldots,p-2\}$. We claim that shifting $w$ by the number whose base $p$ expansion is $\Delta_p = (m+1)^{p-1}0^{\alpha(p-1) + r + \beta}$ (which is a number less than $p^{(2 + \alpha)p + \beta}$) will result in a recurrence of $w$. Let's begin by inspecting $$(n + \Delta)_p = n^*(m+1)0^{(q-\alpha-1)(p-1)}(m+1)^{p-2}(m)(p-1)^{\alpha(p-1) + r}\gamma.$$ First note that, using Fermat's Little Theorem, $b_n\equiv b_{n + \Delta}$. Next note that $\ell < p^{\alpha(p-1) + \beta} - p^\beta$, so $i+\Delta,i+\Delta +1,\ldots, n + \Delta + h - 1$ all have the shared prefix $n^*(m+1)0^{(q-\alpha-1)(p-1)}(m+1)^{p-2}m$ whose contribution is the same as $n^*m$; meanwhile all of $n + \Delta + h,\ldots, i + \Delta + \ell - 1$ all have the shared prefix $n^*(m+1)0^{(q-\alpha-1)(p-1)}(m+1)^{p-1}$ whose contribution is the same as $n^*(m+1)$. Thus, we have a recurrence after at most $p^{(2 + \alpha)p + \beta} = p^{(\alpha - 1)(p-1) + \beta}p^{3p + \alpha - 1}\leq p^{3p + \alpha - 1}(\ell + p^\beta)\leq p^{3p + \alpha + \beta}\ell$ shifts. Therefore, in all of our cases, we have a recurrence of $w$ within at most $C\cdot p^{3p + \alpha + \beta}\ell$ shifts.
\end{proof}

\begin{rem}
$p^\alpha$ is bounded by some constant times $\ell^{\frac{1}{p-1}}$, so in total our recurrence bound is a constant times $\ell^{\frac{p}{p-1}}$. This bound has a much larger constant factor than the one observed in Theorem 5 of \cite{CongruenceAutomaton} ($200\ell$ for the Motzkin numbers mod $5$). Additionally, the bound here is slightly worse than being $O(\ell)$ as is observed in Theorem 5 of \cite{CongruenceAutomaton}, and the author suspects that in the case that one of the first $p$ elements of $a_n$ is a primitive root, then there is an alternative argument that uses inverses in place of Fermat's Little Theorem to achieve an $O(\ell)$ bound.
\end{rem}

Using our main result, we can now draw as corollaries a refinement of Theorem 10 from \cite{CongruenceAutomaton} as well as validate the conjecture of Problem 6 from \cite{CongruenceAutomaton} proving that Burns' observations in \cite{smallprimesmotzkin} hold in general.

\begin{cor}
The central trinomial coefficients mod $p$ are uniformly recurrent if and only if $p$ does not divide any of the central trinomial coefficients (which can be checked for $n < p$). Furthermore if $p$ does divide a central trinomial coefficient, then $0$ has density $1$ in the central trinomial coefficients mod $p$.
\end{cor}\qed
\begin{cor}
The Motzkin numbers are uniformly recurrent mod $p$ if and only if $p$ does not divide any central trinomial coefficients. Furthermore if $p$ does divide a central trinomial coefficient, then $0$ has density $1$ in the Motzkin numbers mod $p$.
\end{cor}
\begin{proof}
The Motzkin numbers satisfy $2M_n = 3T_n + 2T_{n+1} - T_{n+2}$ where $T_n = ct\left[(x^{-1} + 1 + x)^n\right]$ (see the next section or \cite{motzkincentralid}), so the theorem applies for all primes $p > 2$ because $M_n\equiv M_{p,n} = 2^{-1}(3T_n + 2T_{n+1} - T_{n+2})\mod p$. For $p=2$, one can prove uniform recurrence directly from the automaton of figure 1 in \cite{CongruenceAutomaton} which shows that, ignoring the least significant digit, the value of $M_n\mod 2$ is determined by the position of the first (from the right) $0$ in $(n)_2$. So if $w = (M_n\mod 2)(M_{n+1}\mod 2)\cdots(M_{n+\ell - 1}\mod 2)$, then we can let $\Delta$ be one of $2^{\lfloor\log_2\ell\rfloor + 1}$ or $2^{\lfloor\log_2\ell\rfloor + 2}$ and at least one of these will yield $w = (M_{n+\Delta}\mod 2)\cdots(M_{n+\Delta+\ell - 1}\mod 2)$, as desired.\\

Lastly, Proposition \ref{density} applies as well, which completes the corollary.
\end{proof}

The fact that the Motzkin numbers have an identity in terms of the central trinomial coefficients is no coincidence, and we detail this connection in the following section.

\section{A Family of Applicable Sequences}\label{AppSeq}
We now generalize our results for the Motzkin numbers slightly to sequences of the form $ct\left[P(x)^nQ(x)\right]$ where $P$ is of the (symmetric) form $\alpha_1x^{-1} + \alpha_0 + \alpha_1x$ and $Q$ is any Laurent polynomial. For example, $P(x) = x^{-1} + 1 + x$ with $Q(x) = 1-x^2$ gives us the Motzkin numbers (in fact, any symmetric $P$ with this $Q(x)$ gives a weighted Motzkin sequence \cite{WeightedMotzkin}), whereas the same $P$ with $Q(x) = 1-x$ gives the Riordan numbers, A005043 of \cite{oeis}, and $P(x) = x^{-1} + 2 + x^2$ with $Q(x) = 1-x$ gives the Catalan numbers, A000108 of \cite{oeis}.

\begin{prop}
If $a_n = ct\left[P(x)^n\right]$ where $P(x) = \alpha_1x^{-1} + \alpha_0 + \alpha_1x$ and $Q(x)$ is any Laurent polynomial, then for $p>2$, $b_n = ct\left[P(x)^nQ(x)\right]$ is uniformly recurrent mod $p$ if and only if $p$ does not divide any $a_n$ (which can be checked for $n < p$).
\end{prop}
\begin{proof}
In view of Theorem \ref{main}, it will be sufficient to find an integral linear combination of the $a_{n+i}$ which yields a sequence congruent to $b_n$ mod $p$.\\

Let $a_{n,i} = ct\left[P(x)^nx^i\right]$ which is the same as the coefficient on $x^i$ (or $x^{-i}$) in $P(x)^n$ (so $a_{n,0} = a_n$). Notice that $a_{n+i,0} = \sum_{j=-i}^ia_{i,j}\cdot a_{n, j} = a_{i,0}\cdot a_{n,0} + \sum_{j=1}^i2a_{i,j}\cdot a_{n,j}$ (see Figure \ref{fig:triangle} to see where this identity comes from). This along with the fact that $a_{i,i} = \alpha_1^i$ yields $2\alpha_1^i\cdot a_{n,i} = a_{n+i, 0} - a_{i,0}\cdot a_{n,0} - \sum_{j=1}^{i-1}2a_{i,j}\cdot a_{n,j}$. Finally, induction applied to $a_{n,j}$ with $j<i$ using this equality shows that if $p\nmid \alpha_1$, then $a_{n,i}$ can be written as linear combination of $a_{n,0},a_{n+1,0},\ldots,a_{n+i, 0}$ over $\F_p$ (since $2$ and $\alpha_1$ are units). In fact, if $\alpha_1 = 1$ then we even get that $2\cdot a_{n,i}$ can be written outright as an integral linear combination in this way. For example, if $P(x) = x^{-1} + 1 + x$ and $Q(x) = 1-x^2$, we get an identity for the Motzkin numbers in terms of the central trinomial coefficients by finding an identity for $a_{n,2}$ (and $a_{n,0}$) in terms of central coefficients (see the example below).\\

\begin{figure}[h]
    \centering
    \includegraphics[width=0.5\textwidth]{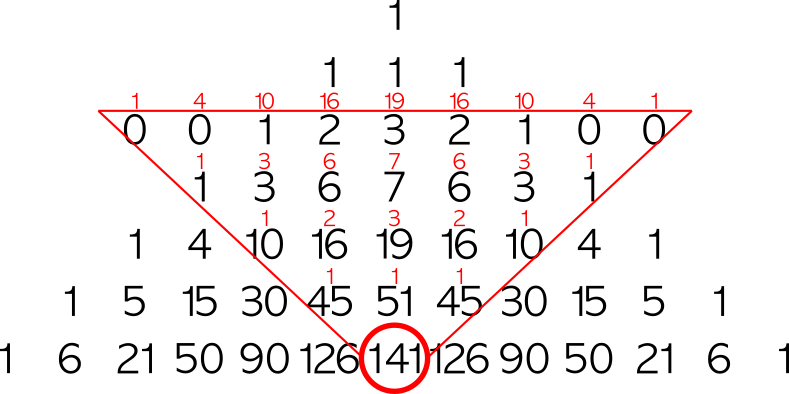}
    \caption{A demonstration of why $a_{n+i,0} = \sum_{j=-i}^ia_{i,j}\cdot a_{n, j}$ when $\alpha_0 = \alpha_1 = 1$. The small red numbers count the number of contributions of each number in a row to the circled $141$.}
    \label{fig:triangle}
\end{figure}

Thus, in the case that $p\nmid \alpha_1$, if $Q(x) = \sum_{j\in\Z} c_jx^j$ then $b_n = ct\left[P(x)^nQ(x)\right] = \sum_{j\in\Z} c_j a_{n,-j}$ which is congruent mod $p$ to a linear combination of $a_{n+i}$'s ($0\leq i\leq \max(\deg Q(x), \deg Q(x^{-1}))$), and so Theorem \ref{main} applies.\\

On the other hand, if $\alpha_1\equiv 0\mod p$, then we simply get $b_n \equiv ct\left[\alpha_0^nQ(x)\right] = \alpha_0^n\cdot ct\left[Q(x)\right]\mod p$, which is periodic (and thus uniformly recurrent).\\

In either case, for any $p>2$, we have that $b_n$ will be uniformly recurrent.
\end{proof}

\begin{example}
Let's show where the identity, $2M_n = 3T_n + 2T_{n+1} - T_{n+2}$, that we have been using to apply our results to the Motzkin numbers comes from. First note that $M_n = ct\left[(x^{-1} + 1 + x)^n(1-x^2)\right] = ct\left[(x^{-1} + 1 + x)^n\right] - ct\left[(x^{-1} + 1 + x)^nx^2\right] = a_{n,0} - a_{n,2}$. Thus, if we let $T_n = ct\left[(x^{-1} + 1 + x)^n\right] = a_{n,0}$, $A_n = ct\left[(x^{-1} + 1 + x)^nx\right] = a_{n,1}$ and $B_n = ct\left[(x^{-1} + 1 + x)^nx^2\right] = a_{n,2}$, then we can use that $T_{n+2} = 3T_n + 4A_n + 2B_n$ and $T_{n+1} = T_n + 2A_n$ (see figure \ref{fig:triangle} for justification) to see that $3T_n + 2T_{n+1} - T_{n+2} = 2(T_n - B_n) = 2M_n$.
\end{example}

\bibliography{UniformRecurrenceInTheMotzkinNumbers}{}

\begin{thebibliography}{10}

\bibitem{motzkincentralid}
E.~Barcucci, R.~Pinzani, and R.~Sprugnoli.
\newblock The {M}otzkin family.
\newblock {\em Pure Math. Appl. Ser. A}, 2(3-4):249--279, 1992.

\bibitem{smallprimesmotzkin}
Rob Burns.
\newblock Structure and asymptotics for motzkin numbers modulo small primes
  using automata, 2016.

\bibitem{deutschsagan}
Emeric Deutsch and Bruce~E. Sagan.
\newblock Congruences for catalan and motzkin numbers and related sequences.
\newblock {\em Journal of Number Theory}, 117(1):191--215, 2006.

\bibitem{motzkinsurvey}
Robert Donaghey and Louis~W Shapiro.
\newblock Motzkin numbers.
\newblock {\em Journal of Combinatorial Theory, Series A}, 23(3):291--301,
  1977.

\bibitem{LucasCongruence}
Joel~A. Henningsen and Armin Straub.
\newblock Generalized {L}ucas congruences and linear {$p$}-schemes.
\newblock {\em Adv. in Appl. Math.}, 141:Paper No. 102409, 20, 2022.

\bibitem{walnut}
Hamoon Mousavi.
\newblock Automatic theorem proving in walnut, 2021.

\bibitem{oeis}
{OEIS Foundation Inc. (2024)}.
\newblock The {O}n-{L}ine {E}ncyclopedia of {I}nteger {S}equences.
\newblock Published electronically at \texttt{http://oeis.org}.

\bibitem{CongruenceAutomaton}
Narad Rampersad and Jeffrey Shallit.
\newblock Congruence properties of combinatorial sequences via {W}alnut and the
  {R}owland-{Y}assawi-{Z}eilberger automaton.
\newblock {\em Electron. J. Combin.}, 29(3):Paper No. 3.36, 13, 2022.

\bibitem{rowlandzeilberger}
Eric~S. Rowland and Doron Zeilberger.
\newblock A case study in meta-automation: automatic generation of congruence
  automata for combinatorial sequences.
\newblock {\em Journal of Difference Equations and Applications}, 20:973 --
  988, 2013.

\bibitem{WeightedMotzkin}
Wen-jin Woan.
\newblock A recursive relation for weighted {M}otzkin sequences.
\newblock {\em J. Integer Seq.}, 8(1):Article 05.1.6, 8, 2005.

\end{thebibliography}
\bibliographystyle{plain}

\end{document}